\documentclass[12pt]{amsart}

\usepackage{latexsym,amssymb,amscd,amsmath,mathrsfs,bbm}

\theoremstyle{plain}
\newtheorem{intro}{}
\newtheorem{theorem}{Theorem}[section]
\newtheorem{prop}[theorem]{Proposition}
\newtheorem{lemma}[theorem]{Lemma}
\newtheorem{corol}[theorem]{Corollary}
\newtheorem{conj}{Conjecture}

\theoremstyle{definition}
\newtheorem{defin}[theorem]{Definition}
\newtheorem{exam}[theorem]{Example}

\def\iff{if and only if }

\def\hot{\mathop\mathrm{Hot}\nolimits}
\def\Es{\mathop\mathrm{Es}\nolimits}
\def\End{\mathop\mathrm{End}\nolimits}
\def\Mat{\mathop\mathrm{Mat}\nolimits}
\def\add{\mathop\mathrm{add}\nolimits}
\def\hos{\mathop\mathrm{Hos}\nolimits}
\def\hhos{\mathop{\widehat{\mathrm{Hos}}}\nolimits}
\def\nil{\mathop\mathrm{nil}\nolimits}
\def\gl{\mathop\mathrm{GL}\nolimits}
\def\sl{\mathop\mathrm{SL}\nolimits}
\def\trs{\mathop\mathrm{tors}\nolimits}

\def\mtr#1{\begin{pmatrix}#1\end{pmatrix}}
\def\dlim{\varinjlim}

\def\lst#1#2{ #1_1 , #1_2 , \dots , #1_{#2} }
\def\set#1{\left\{\,#1\,\right\}}
\def\setsuch#1#2{\left\{\,#1\mid #2\,\right\}}
\def\lb{\textup{(}}	\def\rb{\textup{)}}

\def\wee{\wedge}	\def\8{\infty}	\def\+{\oplus}
\def\*{\otimes}	\def\xx{\times}	\def\bop{\bigoplus}
\def\sb{\subset}	\def\sbe{\subseteq}
\def\sp{\supset}	\def\spe{\supseteq}
\def\xarr{\xrightarrow}	\def\bup{\bigcup}

\def\nF{\mathrm f}	\def\nT{\mathrm t}
\def\mZ{\mathbb Z}	\def\mQ{\mathbb Q}	\def\sA{\mathsf A}
\def\cA{\mathscr A}	\def\cS{\mathscr S}	\def\La{\Lambda}
\def\al{\alpha}		\def\be{\beta}		\def\io{\iota}
\def\Ga{\Gamma}	
	\def\ga{\gamma}	\def\vi{\varphi}

\def\CW{\mathsf{CW}}
\def\hS{\hat\cS}		\def\bve{\bigvee}
\def\Ab{\mathsf{Ab}}	\def\opp{^\mathrm{op}}
\def\hLa{\bar\La}	\def\hM{\bar M}	\def\hN{\bar N}
	\def\tX{\tilde X}	\def\hL{\bar L}
\def\hZ{\hat\mZ}	\def\hX{\hat X}	\def\hY{\hat Y}	
\def\hZ{\hat Z}		
\def\tvi{\tilde\phi}

\begin{document}

\title{On genera of polyhedra}
\author{Yuriy Drozd}
\author{Petro Kolesnyk}
\address{Institute of Mathematics, National Academy of Sciences, 01034 Kyiv, Ukraine}
\email{drozd@imath.kiev.ua}
\urladdr{www.imath.kiev.ua/$\sim$drozd}
\email{iskorosk@googlemail.com}
\subjclass[2010]{55P42, 55P60, 55P10}
\keywords{stable homotopy category, polyhedron, genus, cancellation, orders in semisimple 
algebras}

\begin{abstract}
  We consider the stable homotopy category $\cS$ of \emph{polyhedra} (finite 
cell complexes). We say that two polyhedra $X,Y$ are \emph{in the same genus} and write 
$X\sim Y$ if $X_p\simeq Y_p$ for all prime $p$, where $X_p$ denotes the image of 
$X$ in the localized category $\cS_p$. We prove that it is equivalent to the stable 
isomorphism $X\vee B_0\simeq Y\vee B_0$, where $B_0$ is the wedge of 
all spheres $S^n$ such that $\pi^S_n(X)$ is infinite. We also prove that a stable 
isomorphism $X\vee X\simeq Y\vee X$ implies a stable isomorphism $X\simeq Y$.

\end{abstract}

\maketitle

%\section*{Introduction}

 Genera, i.e. classes of modules which have isomorphic localizations, are widely studied 
and applied in the theory of integral representations, see \cite{cr1,cr2}. On the other 
hand, the notion of genera is natural in other categories, for instance, in the 
\emph{stable homotopy category} \cite{co}, which plays an important role in algebraic 
topology. This paper is an attempt to study genera in the stable homotopy category. If we 
consider finite cell complexes (``polyhedra''), such a study can be reduced to integral 
representations of orders in semisimple algebras. Therefore we can apply a 
deep 
theory developed for genera of integral representations. In particular, we obtain the 
following results:

\begin{intro}
 Two polyhedra $X,Y$ are in the same genus \iff $X\vee B_0$ is stably isomorphic to 
$Y\vee B_0$, where $B_0$ is the wedge of all spheres $S^n$ such that the stable homotopy 
group $\pi^S_n(X)$ is infinite \lb or, the same, not torsion\rb.
\end{intro}

\begin{intro}
 If $X\vee X$ and $Y\vee X$ are stably isomorphic, then $X$ and $Y$ are stably isomorphic 
too.
\end{intro}

 It seems very plausible that other results on genera of integral 
representations also have analogues in the stable homotopy category. In particular, we 
conjecture that the number of stable isomorphism classes in a genus is bounded when we 
consider polyhedra of a prescribed dimension.
We also give some examples of calculating genera for polyhedra of small dimensions.

\section{Stable category and genera}
\label{s1}

 First recall the definitions concerning \emph{stable homotopy category}. We consider the category of connected topological spaces $X$ with fixed points $*_X$. The \emph{homotopy category} $\hot$ has such spaces 
as objects, while $\hot(X,Y)$ is the set of homotopy classes of continuous maps (preserving fixed points). We denote by $X\vee Y$ the \emph{wedge} (one-point union) of the spaces $X$ and $Y$, i.e. the subspace $(X\xx *_Y)\cup(*_X\xx Y)\sb 
X\xx Y$, and by $X\wee Y$ their \emph{smash product}, i.e. $X\xx Y/X\vee Y$. The \emph{suspension of} $X$ is, by definition, the smash product $SX=S^1\wee Y$. We denote by $S^nX$ the $n$-th iterated suspension (isomorphic to $S^n\wee X$). The suspension induces a functor $S:\hos\to\hos$. On the subcategory of simply connected spaces this functor is \emph{conservative}, i.e. $Sf$ is an isomorphism \iff so is $f$. Note that the suspensions are \emph{cogroups} in the homotopy category and their iterations $S^nX$ are commutative cogroups. Therefore, all 
spaces $\hos(S^nX,Y)$ have a natural group structure, commutative if $n>1$. The 
suspension induces group homomorphisms $\hos(S^nX,S^nY)\to\hos(S^{n+1}X,S^{n+1}Y)$ (see \cite{sw} for details). Now we define the \emph{group of stable maps} $\hos(X,Y)$ as the direct limit $\dlim_n\hot(S^nX,S^nY)$. It is a commutative group. If $\al\in\hot(S^nX,S^nY)$, $\be\in\hot(S^mY,S^mZ)$, the product $S^n\be\circ S^m\al$ is 
defined and belongs to $\hot(S^{m+n}X,S^{m+n}Z)$. Its class in $\hos(X,Z)$ is, by definition, the product of the classes of $\al$ and $\be$. Thus we obtain the \emph{stable homotopy category} $\hos$. It is an additive category, where the wedge plays role of the direct sum. Moreover, it is \emph{fully additive}, i.e. any idempotent 
$e\in\hos(X,X)$ splits as $e=\io\pi$ for some morphisms $\pi:X\to Y$ and $\io:Y\to 
X$ such that $\pi\io=1_Y$ (see \cite[page 86]{co}). If $1-e=\io'\pi'$, where $\pi':X\to Y'$, $\io':Y'\to X$ and $\pi'\io'=1_{Y'}$, the morphisms $\io,\pi,\io',\pi'$ define a decomposition $X\simeq 
Y\vee Y'$. We denote by $\Es(X)$ the endomorphism ring $\hos(X,X)$ of $X$ in the stable 
homotopy category, and by $kX$ the wedge of $k$ copies of $X$.

 Consider the full subcategory $\CW\sb\hot$ consisting of \emph{polyhedra}, i.e. finite 
cell complexes. For such polyhedra the direct limit in the definition of $\hos$ actually stabilizes at a finite level. It follows from the Generalized Freudenthal Theorem \cite[Theorem 1.21]{co}.

\begin{theorem}\label{t11}
   If $\,\dim X\le m$ and $\,Y$ is $(n-1)$-connected \lb i.e. $\pi_k(Y)=0$ for $k<n$\rb\ 
then the suspension map $\hot(X,Y)\to \hot(SX,SY)$ is bijective if $m<2n-1$ and surjective if $m=2n-1$.

 In particular, the map $\hot(S^kX,S^kY)\to\hos(X,Y)$ is bijective for $k>m-2n+1$ and surjective for $k=m-2n+1$.
  \end{theorem}
  
  \begin{corol}\label{t12}
 If $X$ is a polyhedron of dimension at most $n$, the map $\hot(S^kX,S^kY)\to\hos(X,Y)$ 
is bijective for $k> n+1$ and surjective for $k=n+1$. In particular, $\pi^S_n(X)\simeq\pi_{2(n+1)}(S^{n+2}X)$.
  \end{corol}

 In what follows, when speaking on polyhedra, we always consider them as objects of 
the stable homotopy category $\cS$. In particular, an \emph{isomorphism} always means a 
\emph{stable isomorphism}. An important feature of the stable homotopy category $\cS$ is 
that all its $\hos$-groups are finitely generated \cite[Corollary X.8.3]{hu}. 

 Let $\CW_n^m$ be the full subcategory of $\CW$ consisting of $(n-1)$-connected polyhedra 
of dimension at most $n+m$. The suspension functor maps $\CW_n^m$ to $\CW_{n+1}^m$. If $n>m+1$ it is an equivalence of categories. If $n=m+1$, it is an \emph{epivalence}, i.e. this functor is full, dense and conservative. In 
particular, it is one-to-one on the isomorphism classes of objects. Set $\CW^m=\bup_{n=1}^\8\CW^m_n$. We denote by $\cS$ the image in $\hos$ of the category $\CW$ and by $\cS^m$ the image of $\CW^m$ in $\cS$.

 Let $\mZ_p=\setsuch{\frac ab}{a,b\in\mZ,\ p\nmid b}$, where $p$  is a prime integer, 
$\hZ_p$ be the ring of $p$-adic integers. We denote by $\cS_p$ ($\hS_p$) the category 
which has the same objects as $\cS$ but the sets of morphisms are 
$\hos_p(X,Y)=\hos(X,Y)\*\mZ_p$ (respectively, $\hhos_p(X,Y)=\hos(X,Y)\*\hZ_p$). 
Actually, $\hos_p(X,Y)$ ($\hhos_p(X,Y)$) coincides with the group of stable maps between 
the $p$-localizations (respectively, $p$-adic completions of $X$ and $Y$) in the 
sense of Artin--Mazur--Sullivan \cite{su}. For the sake of convenience, we 
denote the image in $\cS_p$ ($\hS_p$) of a polyhedron $X$ by $X_p$ (respectively, by 
$\hX_p$). Since all groups $\hos(X,Y)$ are finitely generated for $X,Y\in\cS$,
$\hhos_p(X,Y)$ coincides with the $p$-adic completion of $\hos_p(X,Y)$.  Therefore,
$X_p\simeq Y_p$ in $\cS_p$ \iff $\hX_p\simeq \hY_p$ in $\hS_p$. 

\begin{defin}\label{t13}
  We say that two polyhedra $X$ and $Y$ are \emph{of the same genus} and write $X\sim Y$ if $X_p\simeq Y_p$ for every prime $p$. We denote by $G(X)$ the \emph{genus of the polyhedron} $X$, i.e. the full subcategory of $\cS$ consisting of all polyhedra which are in the same genus as $X$, and by $g(X)$ the number of isomorphism classes in $G(X)$. (Further we will see that it is always finite.)
\end{defin}

 Since all endomorphism rings $\hhos_p(X,X)$ are finitely generated modules over a 
complete local Noetherian ring $\mZ_p$, the \emph{cancellation law} holds in
$\hS_p$ and therefore in $\cS_p$: if $\hX_p\vee\hZ_p\simeq\hY_p\vee\hZ_p$  
(or $X_p\vee Z_p\simeq Y_p\vee Z_p$), then $hX_p\simeq \hY_p$ (respectively, $X_p\simeq 
Y_p$). So we have:

\begin{prop}\label{t14}
  If $X\vee Z\simeq Y\vee Z$, then $X\sim Y$.
\end{prop}

 Later we will show that the converse is also true. It means that the relation $X\sim Y$ is the same as the relation $X\equiv Z$ from \cite[page 90]{co}.

We also consider the category $\cS_\mQ$, having the same objects as $\cS$, but with the 
morphism groups $\hos_\mQ(X,Y)=\hos(X,Y)\*\mQ$. This category is \emph{semisimple}:   
 \[ \hos_\mQ(S^n,S^m)=\begin{cases}
 \mQ &\text{if } m=n,\\
 0 &\text{if } m\ne n
\end{cases} 
\]
 and every object in $\cS_\mQ$ is isomorphic to a direct sum of spheres. Namely,

\begin{prop}\label{t15}
 For any object $X\in\cS$, set $r_n(X)=\dim_\mQ\hos_\mQ(S^n,X)$ and 
$B(X)=\bve_nr_n(X)S^n$. Then $X$ is isomorphic to $B(X)$ in the category $\cS_\mQ$.
\end{prop}

Note that $\hos_\mQ(S^n,X)$ is a finite dimensional vector space over $\mQ$, zero if 
$n>\dim X$, since the stable homotopy groups $\pi^S_n(S^m)$ are periodic for $n>m$ 
\cite{sw}. For any abelian group $A$ we denote by $\trs(A)$ its \emph{torsion part}, i.e. 
the subgroup of all torsion elements. Especially, for objects $X,Y\in\cS$ we denote $\trs(X,Y)=\trs\hos(X,Y)$.
 
\begin{proof}
  We need a definition and an easy lemma. Recall that a \emph{cone sequence} of a map $f:X\to Y$ is the sequence $X\xarr f Y\xarr g C_f\xarr h SX$. Here $C_f$ is the \emph{cone of $f$}, 
i.e. the factorspace $CX\cup Y/\approx$, where $CX=X\xx [0,1]/X\xx1$ is the cone over $X$ and $\approx$ is the equivalence relation such that all nontrivial equivalences are $(x,0)\approx f(x)$. The map $g$ is the obvious embedding $Y\to C_f$ and $h$ is the surjection 
$C_f\to C_f/Y\simeq SX$. Any sequence isomorphic (in $\cS$) to a cone sequence is called 
a \emph{cofibration sequence}.

\begin{defin}
  Let $F$  be an additive functor from $\cS$ to an abelian category $\cA$. We call $F$ \emph{exact} if for every cofibration sequence  $X\to Y\to Z\to SX$ the induced sequence $FX\to FY\to FZ\to F(SX)$ is exact.
\end{defin}

 For instance, every representable functor $\hos(X,\_\,)$ is exact \cite[Theorem 
1.25]{co}. The same is true for the representable contravariant functors $\hos(\_\,,X)$ (considered as functors $\cS\to\Ab\opp$) \cite[page 6, Property 6]{co}.

\begin{lemma}\label{t16}
  Let $F,G$ be exact functors $\cS\to\cA$ and $\phi:F\to G$ be a morphism of functors 
such that $\phi(S^k):F(S^k)\to G(S^k)$ is an isomorphism for each $k$ and if $\phi(X)$ is 
an isomorphism, so is also $\phi(SX)$. Then $\phi$ is an isomorphism of functors.
\end{lemma}
\begin{proof}
  We prove that $\phi(X)$ is an isomorphism by induction on $n=\dim X$. The claim is 
obvious for $n=1$, since every $1$-dimensional polyhedron is isomorphic to a wedge of spheres. Suppose that it holds for 
polyhedra of dimension $n$. Let $\dim X=n+1$, $Y=X^n$ be its $n$-th skeleton. Then there 
is a cofibration sequence $kS^n\to Y\to X\to kS^{n+1}\to SY$. It gives rise to the 
commutative diagram with exact rows
\[
  \begin{CD}
   F(kS^n)@>>> FY @>>> FX @>>> F(kS^{n+1}) @>>> F(SY) \\
   @V\phi(kS^n)VV   @V\phi(Y)VV  @V\phi(X)VV @VV\phi(kS^{n+1})V @VV\phi(SY)V \\
    G(kS^n)@>>> GY @>>> GX @>>> G(kS^{n+1}) @>>> G(SY)\,.
\end{CD}  
\]
 By the condition on $\phi$ and the induction supposition, all vertical morphisms except 
$\phi(X)$ are isomorphisms. Then the 5-Lemma asserts that $\phi(X)$ is an isomorphism too.
\end{proof}

 Let now $X$ be a polyhedron, $r_n=r_n(X),\,B=B(X)$. Choose morphisms $S^n\to X$ such that their images form a basis of $\hos_\mQ(S^n,X)$. We get a morphism $B_n=r_nS^n\to X$ which induces an isomorphism $\hos_\mQ(S^n,B_n)\to\hos_\mQ(S^n,X)$. Altogether, we get a morphism $f:B=B(X)\to X$ which induces isomorphisms $\hos_\mQ(S^n,B)\to\hos_\mQ(S^n,X)$ for each $n$. The morphism $f$ gives rise to the morphism of functors $f_*:\hos_\mQ(\_\,,B)\to\hos_\mQ(\_\,,X)$. Since representable functors are exact and the functor $\_\*\mQ$ is exact in $\Ab$, these 
functors are exact too, and $f_*$ satisfies the conditions of Lemma \ref{t16}. Therefore, 
$f_*$ is an isomorphism of functors. Now the Yoneda Lemma implies that $f:B\to X$ is an 
isomorphism in $\cS_\mQ$.\end{proof}

\begin{corol}\label{t17}
 \begin{enumerate}
\item    There are morphisms $\al:X\to B(X)$ and $\be:B(X)\to X$ in $\cS$ such that $\al\be=t1_{B(X)}$ and $\be\al=t1_X$ for some integer $t>0$.
\item   $\dim_\mQ(X,S^n)=r_n(X)$.
\item   If $\La=\Es (X)$, then $\La\*\mQ\simeq\prod_n\Mat(r_n(X),\mQ)$.
\end{enumerate}
\end{corol}

 \begin{defin}\label{t18}
  Let $\cA$ be a preadditive category. A morphism $f:X\to Y$ in $\cA$ is said to be \emph{essentially nilpotent} if for every morphism $g:Y\to X$ the product $fg$ (or, equivalently, $gf$) is nilpotent. We denote the 
set of such morphisms by $\nil(X,Y)$ and by $\nil(X)$ if $X=Y$. The class $\nil\cA=\bup_{X,Y}\nil(X,Y)$ of all essentially nilpotent morphisms is called the \emph{nilradical} of the category $\cA$.\!%
\footnote{\, If $\cA$ only has one object, i.e. is actually a ring, $\nil\cA$ is its 
\emph{upper nil radical} in the sense of \cite{jac}.}
 It is an ideal in $\cA$, so the factor category $\cA^0=\cA/\nil\cA$ is defined. We call it the \emph{semiprime part} of $\cA$. It contains no essentially nilpotent morphisms and a morphism $f:X\to Y$ is an isomorphism \iff so is its image in $\cA^0$. In particular, two objects $X,Y$ are isomorphic in $\cA$ \iff they are isomorphic in $\cA^0$.

In particular, for the category $\hos$, we write
\begin{align*}
 \hos^0(X,Y)&=\hos(X,Y)/\nil(X,Y),\\
 \Es^0(X)&=\Es(X)/\nil(X).
\end{align*}
\end{defin}

\begin{corol}\label{t19}
 $\nil(X,Y)\sbe\trs(X,Y)$ for any objects $X,Y\in\cS$. 
\end{corol}
 \begin{proof}
  Otherwise $\Es(X\vee Y)\*\mQ$ contains nilpotent ideals.
\end{proof}

\begin{prop}\label{t10}
  Let $\La$ be a semiprime ring \lb i.e. a ring without nil-ideals\rb\ such that 
$\trs\La$ is a finitely generated group. Then $\La=\trs\La\xx \La^{\nF}$ for some torsion free ideal $\La^{\nF}$ \lb isomorphic to $\La/\trs\La$ as a ring\rb. 
\end{prop}
\begin{proof}
 Obviously, $\trs\La$ is an ideal. Since its additive group is finitely generated, it is 
of finite length as $^\La$-module. Therefore, it contains a minimal right ideal $I_1$. 
Since $I_1$ is not nilpotent, it is generated by an idempotent: $I_1=e_1\La$ where 
$e_1^2=e_1$. Hence $\trs\La=I_1\+I'$ for some right ideal $I'$. The same observation 
shows that $I'=e_2\La\+I''$ and so on. Eventually we get $\trs\La=e\La$ for some 
idempotent $e$. Since $(1-e)\La e\sbe\trs\La$, we get $(1-e)\La e\sbe e\La$, so $(1-e)\La 
e=0$. Then $e\La(1-e)$ is a nilpotent ideal, so $e\La(1-e)=0$ too, and 
$\La=e\La\+(1-e)\La$, where both summands are two-sided ideals, which implies the 
statement.
\end{proof}

\begin{corol}\label{t1a}
  Every object $X\in\cS$ splits uniquely as $X\simeq X^{\nT}\vee X^{\nF}$, so that 
$\Es(X^\nF)/\nil(X^\nF)$ is torsion free, while $\Es(X^\nT)$ is torsion. Especially, if 
$X$ is indecomposable in the category $\cS$, then either $\Es(X)$ is torsion or 
$\Es(X)/\nil(X)$ is torsion free. \emph{(The latter condition means that 
$\trs(X)=\nil(X)$.)}
\end{corol}

  We call $X^\nT$ the \emph{torsion part} of $X$ and $X^\nF$ its \emph{torsion reduced} 
part. If $X=X^\nT$, we call $X$ \emph{torsion}; if $X=X^\nF$, we call it \emph{torsion  
reduced}.

\begin{proof}
 Let $\La=\Es(X)$, $\hLa=\Es^0(X)$.  According to Proposition~\ref{t10}, there are 
central idempotents $e_1,e_2$ such that $\hLa=e_1\hLa\+e_2\hLa$, where $e_1\hLa$ is 
torsion and $e_2\hLa$ is torsion free. 
Idempotents modulo $\nil(X)$ can be lifted to idempotents in $\La$, so 
$\La=e^*_1\La\+e^*_2\La$ where $e_i=e_i^*\!\mod\!\nil(X)$. It gives rise to a 
decomposition 
$X=X_1\vee X_2$ of $X$ so that $\Es(X_i)\simeq e^*_i\La e^*_i$. Since $\nil(e\La 
e)=e(\nil\La)e$, we get the statement if we set $X^\nT=X_1$ and $X^\nF=X_2$.
\end{proof}

 \begin{prop}\label{t1b}
\begin{enumerate}
 \item
  For any polyhedra $X,Y$,  $X\sim Y$ \iff $\,X^\nF\sim Y^\nF$ and $\,X^\nT\simeq Y^\nT$.
\item  If $X_1\vee Y_1\simeq X_2\vee Y_2$, where $X_i$ are torsion reduced and $Y_i$ are torsion, then $X_1\simeq X_2$ and $Y_1\simeq Y_2$.
\item  $g(X)=g(X^\nT)$.
\end{enumerate}
\end{prop}
\begin{proof}
 (1) Obviously, every morphism $f:X_p\to Y_p$ maps $X^\nT_p\to Y^\nT_p$ and if $f$ is an isomorphism, so is its restriction on $X^\nT_p$. But $\hos(Z,X^\nT)$ is always torsion, hence, isomorphic to $\bop_p\hos_p(Z,X^\nT)$. Therefore, if $X^\nT_p\simeq Y^\nT_p$ for all $p$, also $X^\nT\simeq Y^\nT$. Thus if $X\sim Y$, then 
$X^\nT\simeq Y^\nT$. Since cancellation holds in $\cS_p$, also $X^\nF\sim Y^\nF$. The converse is evident. 

 (2) Proposition~\ref{t10} implies that $\hos(X_i,Y_j)$ and $\hos(Y_j,X_i)$ belong to 
$\nil(\hos)$. Hence, any morphism from $\hos^0(X_i\vee Y_i,X_j\vee Y_j)$ is given 
by a 
diagonal matrix $\mtr{\al&0\\0&\be}$, where $\al:X_i\to X_j$ and $\be:Y_i\to Y_j$. It is invertible \iff so are $\al$ and $\be$. Since a morphism from $\hos$ is 
invertible \iff its image in $\hos^0$ is invertible, it proves the statement.

 (3) follows immediately from (1) and (2).   
\end{proof}

 \section{$G(X)$ and $G(\La)$}
 \label{s2}

 We are going to establish relations between genera of polyhedra and of modules. First recall a fact from general nonsense. For an object $X$ of a fully additive category $\cA$ we denote by $\add X$ the full subcategory of $\cA$ consisting of all objects that are isomorphic to direct summands of direct multiples $kX$ of the object 
$X$. 

\begin{prop}\label{t21}
  Let $\La=\End X$ be the ring of endomorphism of an object $X$ from a fully additive 
category $\cA$. Then the functor $X^*:Y\mapsto\cA(X,Y)$ establishes an equivalence $\add 
X\simeq\add\La$, where $\La$ is considered as an object of the category of right 
$\La$-modules.\end{prop}
 
 Note that $\add\La$ is actually the category of finitely generated projective right 
$\La$-modules.

\begin{proof}
  The Yoneda Lemma claims that $B\mapsto\cA(\_\,,Y)$ is an equivalence between $\add X$ and 
the category of representable functors $\add X\to\Ab$. But since every object from $\add 
X$ is a direct summand of 
$kX$, any additive functor from $\add X$ is completely defined (up to isomorphism) by its 
value at $X$. It implies the statement.
\end{proof}

 We will use this result in case when $\cA=\cS$, so $\La=\Es(X)$ is the ring of stable 
endomorphisms of $X$. Especially, we will use it to study genera.
 Recall that two $\La$-modules $M,N$ \emph{are in the same genus} \cite{cr1} if 
$M_p\simeq N_p$ for all prime $p$, where $M_p=M\*\mZ_p$. Then we write $N\sim M$. Just as 
above, we denote by $G(M)$ the category of all $\La$-modules $N$ such that $N\sim M$ and 
by $g(M)$ the number of isomorphism classes in $G(M)$. Let $\hLa=\La/\trs\La$. It is an 
order in the semisimple algebra $\prod_i\Mat(r_i(X),\mQ)$ by Corollary~\ref{t17}\,(3). 
Obviously, $g(\La)=g(\hLa)$. Therefore, $g(\La)<\8$ by Jordan--Zassenhaus Theorem 
\cite[Theorem 24.1]{cr1}. 

\begin{prop}\label{t22}
  If $\,Y\sim X$, then $Y\in\add X$.
\end{prop}
\begin{proof} 
 Note that $\hos_p(X,Y)/p\hos_p(X,Y)\simeq\hos(X,Y)/p\hos(X,Y)$ and an endomorphism of 
$X_p$ is invertible \iff it is invertible modulo $p$. Thus if $X_p\simeq Y_p$, there are morphisms $f_p:Y\to X$ and $g_p:X\to Y$ such that $g_pf_p$ does not belong to any maximal ideal $M\sb\Es(Y)$ such that $M\spe p\Es(Y)$. Since any 
maximal ideal of $\Es (Y)$ contains some $p\Es(Y)$, the set $\set{g_pf_p}$ generates the unit ideal. It implies that there are morphisms $f_i:Y\to X$ and $g_i:X\to Y$ ($i=1,2,\dots,k$) such that $\sum_{i=1}^kg_if_i=1$. Let $f:Y\to kX$ be the morphism with the components $\lst fk$ and $g:kX\to Y$ be the morphism with the components $\lst gk$. Then $gf=1_Y$, which means that $Y$ is a direct summand of $kX$.
\end{proof}

\begin{corol}\label{t23}
  $G(X)\simeq G(\La)$, in particular, $g(X)=g(\La)<\8$.
\end{corol}

Now we can apply the known facts from the theory of integral representations to genera 
of polyhedra.

\begin{theorem}\label{t24}
 Let $X,Z$ be two polyhedra such that $Z\in\add X$ \lb for instance, $Z\sim kX$ for some 
$k$\rb. If $X\vee Z\simeq Y\vee Z$, then $X\simeq Y$.
\end{theorem}
\begin{proof}
 Proposition~\ref{t1b} implies that we can suppose $X$ torsion reduced. Then so are also 
$Y$ and $Z$. We consider the functor $X^*$ of Proposition~\ref{t21}. Set 
$\La=\Es(X)=X^*(X),\  M=X^*(Y),\ N=X^*(Z)$. Then $\La\+N\simeq M\+N$ and $N\in\add\La$. 
We have to show that $M\simeq\La$. Set also $R=\nil\La,\ \hLa=\La/R,\ \hM=M/MR,\ 
\hN=N/NR$. Note that $\hLa$ is torsion free as an abelian group. Since $M$ is a 
projective $\La$-module, it is enough to prove that $\hM\simeq\hLa$. By 
Corollary~\ref{t17}\,(3), $\mQ\*\hLa$ is a product of full matrix rings over $\mQ$, so it 
satisfies the \emph{Eichler condition} in the sense of \cite[\S\,51A]{cr2}. Therefore we 
can apply the Jacobinski cancellation theorem \cite[Theorem 51.24]{cr2}: if 
$A\+k\hLa\simeq B\+k\hLa$ for some locally free $\hLa$-modules $A,B$, then $A\simeq B$. 
Since $\hM\+\hN\simeq\hLa\+\hN$, where $\hN\in\add\hLa$, i.e. $\hN\+N'\simeq k\hLa$ for 
some $N'$ and $k$, and $\hM_p\simeq\hLa_p$ for all $p$, it implies that $\hM\simeq\hLa$, 
wherefrom $M\simeq\La$ and $Y\simeq X$.
\end{proof}

Recall that we have set $r_n(X)=\dim_\mQ \hos_\mQ(S^n,X)$. Set 
$$B_0(X)={\bve\hskip-2pt}_{\mbox{\footnotesize$r_n(X)\ne0$}}\hskip1pt S^n.$$ 
In other words, $B_0(X)$ is the wedge of all spheres $S^n$ such that the stable homotopy 
group $\pi^S_n(X)$ is not torsion, or, the same, is infinite.

\begin{theorem}\label{t25}
 $X\sim Y$ \iff $X\vee B_0(X)\simeq Y\+B_0(Y)$.
\end{theorem}
\begin{proof}
  We have already seen (Proposition~\ref{t14}) that $X\vee B_0(X)\simeq Y\+B_0(Y)$ 
implies $X\sim Y$. To show the converse, we can suppose $X$ and $Y$ torsion reduced. Let 
$B_0=B_0(X)$, $\tX=X\vee B_0$. Then $X,Y,B_0\in\add\tX$. Set $\La=\Es(\tX)$, and apply 
the functor $\tX^*$ of Proposition~\ref{t21}. Denote $M=\tX^*(X),\ N=\tX^*(Y),\ 
L=\tX^*(B_0)$; $R=\nil\La,\ \hLa=\La/R,\ \hM=M/MR,\ \hN=N/NR,\ \hL=L/LR$. Since $X\simeq 
B(X)$ in the category $\cS_\mQ$, $\,\mQ\*L$ is a faithful $\mQ\*\La$-module. Hence $\hL$ 
is a faithful $\hLa$-module. Now we can use a Roiter's theorem on genera \cite[Theorem 
31.28]{cr1}. It claims that there is a $\hLa$-module $L'$ such that $L'\sim 
\hL$ and $\hM\+\hL\simeq \hN\+L'$. Note now that $\End_{\hLa}(\hL)\simeq 
\La_0/\nil\La_0$, where $\La_0=\Es(B_0)$, thus $\End_{\hLa}(\hL)\simeq\mZ^m$, where 
$m=\#\setsuch{n\in\mZ}{r_n(X)\ne0}$. Therefore, $g(\hL)=1$ and $L'\simeq\hL$, that is 
$\hM\+\hL\simeq\hN\+\hL$, so $M\+L\simeq N\+L$ and $X\vee B_0\simeq Y\vee B_0$.
 \end{proof}

 Finally, we propose the following conjecture, which is, in some sense, an analogue of 
another Roiter's theorem on genera \cite[Theorem 31.34]{cr1} (boundedness of $g(M)$ for 
all modules $M$ over an order in a semisimple algebra).

\begin{conj}
 For every positive integer $n$ there is an integer $c_n$ such that $g(X)\le c_n$ for 
every polyhedron $X$ of dimension $n$.
\end{conj}

\section{Calculations and examples}
\label{s3}

 For calculation of $g(X)$ for concrete polyhedra $X$ the following facts are 
useful. 

\begin{prop}[\cite{da}] \label{p31}
 Let $\La$ be an order in a semisimple $\mQ$-algebra, $\Ga$ be a maximal order containing 
$\La$ and $\La\spe m\Ga$ for some integer $m>1$. Then $g(\La)=g(\Ga)g(\La,\Ga)$, where 
$g(\La,\Ga)$ is the number of double cosets \[
   \Ga^\times \backslash \prod_{p\mid m} \Ga_p^\times / \prod_{p\mid m}\La_p^\times.
\]
\end{prop}

\begin{prop}\label{p32}
  Let $\La$ be an order in a semisimple $\mQ$-algebra, $\Ga$ be a maximal order containing 
$\La$ and $\La\spe m\Ga$ for some integer $m>1$. Then $g(\La,\Ga)$ equals the 
number of cosets
\[
   \mathop\mathrm{Im} \ga  \backslash (\Ga/m\Ga)^\times / (\La/m\Ga)^\times,
\]
 where $\ga$ is the natural map $\Ga^\times\to(\Ga/m\Ga)^\times$.
\end{prop}
\begin{proof}
 It is evident, since $\La_p^\times\spe 1+m\Ga_p$ and $\Ga_p^\times/(1+m\Ga_p)\simeq 
(\Ga_p/m\Ga_p)^\times$.
\end{proof}

\begin{corol}\label{c33}
  Let $\La$ be an order in $\prod_{i=1}^k\Mat(r_i,\mQ)$ such that $\Ga\spe\La\spe m \Ga$ 
for some maximal order $\Ga$ and some integer $m>1$. Then $g(\La)=1$ if $m=2$ and $ 
g(\La)\le(\vi(m)/2)^k$ if $m>2$, where $\vi(m)$ is the Euler function.

\emph{ In particular, $g(\La)=1$ if $m\in\set{2,3,4,6}$.}
\end{corol}
\begin{proof}
 We can and will identify $\Ga$ with $\prod_{i=1}^k\Mat(r_i,\mZ)$. Then $g(\Ga)=1$,
\begin{align*}
 (\Ga/m\Ga)^\times&\simeq \prod_{i=1}^k\gl(r_i,\mZ/m),\\
\intertext{and}
 \Ga^\times &=\prod_{i=1}^k\gl(r_i,\mZ).
\end{align*}
 The group $\gl(r_i,\mZ)$ contains all elementary matrices, hence its image in 
$\gl(r_i,\mZ/m)$ contains the subgroup generated by elementary matrices, which is 
$\sl(r_i,\mZ/m)$. The quotient group $\gl(r_i,\mZ/m)/\sl(r_i,\mZ/m)$ is isomorphic to 
$(\mZ/m)^\times$, the isomorphism being induced by the determinate. Moreover, 
$\pm1=\det\ga$ for a diagonal matrix from $\Mat(r_i,\mZ)$. Therefore,
\[
 g(\La)\le \prod_{i=1}^k(\mZ_m^\times:\set{\pm1})=\begin{cases}
  1 &\text{ if } m=2,\\
 (\vi(m)/2)^k &\text{ if } m>2.
\end{cases}
\]
\end{proof}

 Applied to polyhedra, it gives the following result.

\begin{theorem}\label{t34}
 Let $X$ be a polyhedron, $B=\bve_{i=1}^kr_iS^{n_i}$ with different $\lst nk$. Suppose 
that there are maps 
$X\xarr\be B\xarr\al X$ such that $\al\be\equiv m1_X\hskip-1ex\mod\hskip-.5ex\trs(X)$ and 
$\be\al\equiv m1_B\hskip-1ex\mod\hskip-.5ex\trs(B)$ for some integer $m>1$. Then 
$g(X)=1$ if $m=2$ and $ g(X)\le(\vi(m)/2)^k$ if $m>2$. \end{theorem}

\begin{proof}
 Denote $\Ga=\Es^0(B),\ 
\La=\Es^0(X)$. Then $g(X)=g(\La)$, $\Ga$ and $\La$ are orders in the 
semisimple $\mQ$-algebra $\sA=\prod_{i=1}^k\Mat(r_i,\mQ)$, and 
$\Ga\simeq\prod_{i=1}^k\Mat(r_i,\mZ)$ is a maximal order with $g(\Ga)=1$. Define a map 
$\phi:\Ga\to\La$ induced by the map $\Es(B)\to\Es(\La),\ f\mapsto \al f\be$. It is 
injective, additive and $\phi(f)\phi(g)=m\phi(fg)$. Thus the assertion of the theorem 
follows from Corollary \ref{c33} and the following lemma.

 \begin{lemma}\label{l35}
 Let $\La$ be an order in a semisimple $\mQ$-algebra $\sA$, $\Ga$ be a maximal order in  
$\sA$ and $\phi:\Ga\to\La$ be an injective homomorphism of additive groups such that 
$m\phi(ab)=\phi(a)\phi(b)$ for every $a,b\in\Ga$ and some integer $m>1$. Then 
 there is a maximal order $\Ga'$ in $\sA$ such that $\Ga'\spe\La\sp m\Ga'$.
\end{lemma}
\noindent
\textit{Proof.}
 We extend $\phi$ to a homomorphism of $\mQ$-modules $\tvi:\sA\to \sA$. Then 
$m\tvi(ab)=\tvi(a)\tvi(b)$ for any $a,b\in \sA$, so $\psi=\frac{\tvi}m$ is an 
automorphism of the algebra $\sA$ and $\psi(\Ga)=\Ga_1$ is a maximal order in $\sA$ too. 
Moreover, $m\Ga_1=\phi(\Ga)\sbe\La$. Consider the $\La$-$\Ga_1$-bimodule 
$M=m\La\Ga_1\sb\sA$ and set $\Ga'=\setsuch{a\in\sA}{aM\sbe M}\simeq\End_{\Ga_1}M$. 
Obviously, $\Ga'\spe\La$ and $m\Ga'\sbe \Ga'M=M\sb\La$. Since $M$ is right torsion free 
$\Ga_1$-module and $\Ga_1$ is maximal, $\Ga'$ is also maximal \cite[Theorem 26.25]{cr1}.
\end{proof}

 Theorem \ref{t34} implies that Conjecture 1 above follows from the following.

\begin{conj}
 For every integer $d$ there is an integer $m>0$ such that for every polyhedron $X$ of 
dimension $d$ there are maps $X\xarr\be B\xarr\al X$, where $B$ is 
a wedge of 
spheres, such that $\al\be\equiv m1_X\hskip-1ex\mod\hskip-.5ex\trs(X)$ and $\be\al\equiv 
m1_B\hskip-1ex\mod\hskip-.5ex\trs(B)$.
\end{conj}

 In the following examples we use definitions and calculations from \cite[Section 3]{dr}. 
 In particular, we denote by $a$ the $a$-th multiple of a generator of 
the group $\pi_n(S^n)\simeq\mZ$ and by $\eta$ the nonzero element of 
$\pi_n^S(S^{n-1})\simeq\mZ/2$. 
 We also denote by $\mZ\times_m\mZ$ the subring of $\mZ\times\mZ$ consisting of all pairs 
$(\al,\be)$ with $\al\equiv\be\!\pmod m$. Note that $\mZ\times_m\mZ\spe 
m(\mZ\times\mZ)$, so $g(\mZ\times_m\mZ)$ equals the number of double cosets
\[
 \{\pm1\} \times \{\pm1\} \backslash \mZ_m^\times\times\mZ_m^\times / \mZ_m^\times
\]
 under the diagonal embedding of $\mZ_m^\times$ into $\mZ_m^\times\times \mZ_m^\times$. 
It easily gives $g(\mZ\times_m\mZ)=\vi(m)/2$.

 \begin{exam}\label{x36}
 \begin{enumerate}
\item 
 If $M^{n+1}(a)$ is a \emph{Moore atom}, i.e. is defined by the cofibration sequence
\[
  S^n\xarr{a}  S^n \to M^{n+1}(a) \to S^{n+1},
\]
its endomorphism ring is torsion, therefore $g(M^n)=1$.

\item
 The same holds for the Chang atom $C^n(2^r\eta2^s)$ defined by the cofibration 
sequence
\[ \hspace*{3em}
 S^{n-1}\vee S^n\xarr{\mtr{2^r&\eta\\0&2^s}} S^{n-1}\vee S^n \to 
C^{n+1}(2^r\eta2^s)\to S^n\vee S^{n+1},
\] 
 which also has torsion endomorphism ring.

\item
 For Chang atoms $C^{n+1}(2^r\eta)$ and $C^{n+1}(\eta2^s)$ defined, respectively, by the 
cofibration sequences
\begin{align*}
 S^{n-1}&\vee S^n \xarr{\mtr{2^r&\eta}} S^{n-1}\to C^{n+1}(2^r\eta)\to S^n\vee 
S^{n+1}\\\intertext{and}
 &S^n \xarr{\mtr{\eta\\2^s}} S^{n-1}\vee S^n\to C^{n+1}(\eta2^s)\to S^{n+1}
\end{align*}
 the calculations of \cite[Section 3]{dr} show that
\begin{align*}
 \Es^0(C^{n+1}(2^r\eta))&\simeq \Es^0(C^{n+1}(\eta2^s))\simeq\mZ, \\
\intertext{wherefrom}
 g(C^{n+1}(2^r\eta))&= g(C^{n+1}(\eta2^s))=1.
\end{align*}

\item
 Let $C^{n+1}(\eta)$ be the Chang atom defined by the cofibration sequence
\[
 S^n\xarr\eta S^{n-1} \to C^{n+1}(\eta)\to S^{n+1}. 
\]
 One easily verifies that $\Es^0(C^{n+1}(\eta))\simeq\mZ\times_2\mZ$, wherefrom 
$g(C^{n+1}(\eta))=1$. The same holds for the \emph{double Chang atom} $C(\eta^2)$ from 
\cite[Section 5]{dr} defined by the cofibration sequence
\[
 S^n\xarr{\eta^2} S^{n-2} \to C^{n+1}(\eta)\to S^{n+1}, 
\]
 where $\eta^2$ is the nonzero element of $\pi^S_n(S^{n-2})\simeq\mZ/2$.

\item
 Finally, we consider the atom $A^{n+1}(v)$ defined by the cofibration sequence
\[
  S^n\xarr{v\nu} S^{n-3} \to C^{n+1}(\eta)\to A^{n+1}(V)\to S^{n+1},   
\]
 where $0<v\le 12$ and $\nu$ is a generator of the group $\pi_n^S(S_{n-3})\simeq \mZ/24$. 
It follows from \cite[Theorem 2.4]{dr} that $\Es^0(A^{n+1}(v))$ is isomorphic to the 
ring of pairs $(\al,\be)$, where $\al,\be\in\mZ$ and $\al v\nu=\be v\nu$, that is 
$\al\equiv\be\!\pmod m$, where $m=24/d$ and $d=\gcd(v,24)$. It is the ring 
$\mZ\times_m\mZ$. Therefore
\[
 g(A(v))=\begin{cases}
  4 &\text{ if } d=1,\\
  2 &\text{ if } d=2 \text{ or } d=3,\\
  1 &\text{ if } d>3.
\end{cases} 
\]
 Actually, all atoms $A^{n+1}(v)$ with fixed $\gcd(v,24)$ are in the same genus. One 
easily verifies that $B_0(A^{n+1}(v))=S^{n-3}\vee S^{n+1}$, so
\[
 A^{n+1}(v)\vee S^{n-3}\vee S^{n+1}\simeq A^{n+1}(v')\vee S^{n-3}\vee S^{n+1}
\]
 if $\gcd(v,24)=\gcd(v',24)$. Indeed, one can even check that in this case already 
$A^{n+1}(v)\vee S^{n-3}\simeq A^{n+1}(v')\vee S^{n-3}$, as well as $A^{n+1}(v)\vee 
S^{n+1}\simeq A^{n+1}(v)\vee S^{n+1}$.
\end{enumerate}

\end{exam}

\end{document}